\documentclass[12pt]{amsart}
\usepackage{amssymb} 

\newtheorem{thm}{Theorem}[section]
\newtheorem{cor}[thm]{Corollary}
\newtheorem{lemma}[thm]{Lemma}
\newtheorem{prop}[thm]{Proposition}

\newtheorem{conj}[thm]{Conjecture}
\theoremstyle{definition}
\newtheorem{defn}[thm]{Definition}
\newtheorem{rem}[thm]{Remark}

\newtheorem{que}[thm]{Question}

\numberwithin{equation}{section}

\begin{document}
\title[On coverings of groups by normalizers]
{On coverings of groups by normalizers}%
\author[ Amiri, Haji and Jafarian Amiri   ]{m. Amiri, S. Haji and S. M.  Jafarian Amiri   }
\address{ Departamento de  Matem\'{a}tica, Universidade Federal do Amazonas, Department of Mathematics, Faculty of Sciences, University of Zanjan, P.O.Box 45371-38791, Zanjan, Iran}%
\email{mohsen@ufam.edu.br}
\email{haji$ _{-} $saeid@znu.ac.ir}
\email{sm$_{-}$jafarian@znu.ac.ir}%
\email{}
\thanks{}
\thanks{}
\maketitle

\begin{abstract}
For a group $G$, a {\it normalizer covering} of $G$ is a finite set of proper normalizers of some subgroups of $G$ whose union is $G$. We study $p$-groups ($p$ a prime) without  a normalizer covering.
As an application, we determine some non-nilpotent groups having a nilpotent normalizer covering.


\end{abstract}
\maketitle
\section{\bf Introduction and main results }

 Throughout this paper $G$ is a group. Let  $\mathcal{C}=\{A_1, A_2,\cdot, A_m\}$ be a finite collection of proper subgroups of $G$. The set $\mathcal{C}$ is called a {\it cover} or {\it covering} of $G$ if  $G=\cup_{i=1}^mA_i$. 
 
 
 
In 1994, Cohn \cite{cohn} defined $\sigma(G)$, called covering number of $G$, the minimal size of a covering of $G$ and the cover of size $\sigma(G)$ is said to be a {\it minimal cover}. The parameter $\sigma(G)$ has been extensively studied (see \cite{int} and the references therein). For covering of other algebraic structures, see \cite{jour}.



Denote by $\sigma_a(G)$ the minimal size of a covering consisting of abelian subgroups of $G$ and such covering is called a {\it minimal abelian covering} of $G$. This parameter was investigated by some authors (see \cite{mas}, \cite{bro}, \cite{pyb} and \cite{pod}).

 
 Recently the second and third  authors of this article in \cite{haji} introduced $\sigma_c(G)$ the minimal size of a covering consisting of proper centralizers of $G$. Also Atanasov et al. \cite{power} study $\sigma_P(G)$ for Dihedral 2-groups where $\sigma_P(G)$ is the minimal size of a covering consisting of powerful subgroups.

 In this article, we focus on groups having a {\it normalizer covering } which is a  covering, say $\Gamma$, whose members are normalizers of some subgroups i.e. $$\Gamma=\{N_G(H_1),\cdots, N_G(H_{m})\}$$ where $N_G(H_i)=\{g\in G:~g^{-1}H_ig=H_i\}$ for some subgroups $H_i$ of $G$.

We define $\sigma_{\mathfrak{n}}(G)$ the minimal size of a normalizer covering of $G$ and call it the {\it normalizer covering number}. If $G$ is a Dedekind group (all its subgroups are normal), then $G$ does not have a normalizer covering. Note that if $G$ does not have any normalizer covering, then we define $\sigma_{ \mathfrak{n}}(G)=\infty$,  for example $\sigma_{\mathfrak{n}}(D_{2n})=\infty$ where $D_{2n}$ is the dihedral group of order $2n$.

\begin{rem} 
It suffices to restrict our attention to finite groups when determining {\it normalizer covering} numbers, since, by a
result of B. H. Neumann \cite{nu}, a group admits a cover if and only if it has a finite noncyclic homomorphic image. So in the sequel $G$ is finite and $p$ is a prime divisor of $|G|$.

\end{rem}
Note that $\sigma_{\mathfrak{n}}(G)=\sigma(G)$ for every non-abelian simple group $G$. Cohn showed that if $G$ is a non-cyclic $p$-group, then $\sigma(G)=p+1$ (see Theorem 2 of \cite{cohn}). In this article, we prove that if $G$ is a $p$-group without any normalizer covering then either  $G$ is a metacylic group for $p>2$, or $p=2$ and $G$ is a product of two metcyclic subgroup. There exists a $2$-group $G=AllSmallGroups(2^5,IsAbelian,false)[44]$  in the software $GAP$ \cite{gap} such that $\sigma_{ \mathfrak{n}}(G)=5$.  Tomkinson \cite{tom} have showed that there is no group $G$ with $\sigma(G)=7$, whereas it is easy to see that $\sigma_{ \mathfrak{n}}(S_4)=7$ where $S_4$ is the symmetric group on four letters. The next integer which is not a covering number is 11 (see \cite{11}) and  the interested reader can see \cite{int} for integers less than 129 which are not covering numbers. However we can not find any group $G$ with $\sigma_{ \mathfrak{n}}(G)=11$.


 We suggest that it might be of interest to investigate which positive integers$>2$ can or can not be as  {\it normalizer covering} numbers of groups. For any prime $p$ and positive integer $k$, we construct a group $G$ with $\sigma_{ \mathfrak{n}}(G)=p^{k}+1$.
 Finelly, we pose the  following  conjecture:

 \begin{conj}
     Let $G$ be finite $p$-group with a normalizer covering:

     (a) If $p>2$, then $\sigma_{ \mathfrak{n}}(G)=p+1$.

     (b) If $p=2$, then $\sigma_{ \mathfrak{n}}(G)\in \{3,5\}.$
 \end{conj}


 Throughout this paper all notations are usual. For example, $G'$ and $Z(G)$ denote the commutator subgroup and the center of a group $G$, respectively, $C_n$ stands for the cyclic group of order $n$. Moreover $\varPhi(G)$ denotes the Frattini subgroup of $G$. We denote the Fitting subgroup of $G$ by $Fit(G)$. The term {\it normalizer} (except in normalizer covering) always means the normalizer of some subgroup of the group. The rest of the notation is standard and can be found mainly in \cite{Rob}. 

\section{Preliminaries results}
We begin with two simple lemmas but frequently used fact.
\begin{lemma}\label{nocover}
Let $G$ be a group. Then the following statements are equivalent:

(i)$\sigma_{ \mathfrak{n}} (G)=\infty$.

(ii)there exists $x\in G$ whenever $x\in N_G(H)$ for some subgroup $H$ of $G$, then $H\unlhd G$.
\end{lemma}
\begin{proof}
The proof is clear.
\end{proof}

\begin{lemma}\label{norm}
Let $ G $ be a group and $ N\unlhd G $. If $ \sigma_{ \mathfrak{n}}(\frac{G}{N})<\infty$ , then $ \sigma_{ \mathfrak{n}}(G)\leq \sigma_{ \mathfrak{n}}(G/N) $.
\end{lemma}

\begin{proof}
Let  $ \sigma_{ \mathfrak{n}}(G/N)=m $ and $ G/N=\bigcup_{i=1}^m N_{G/N}(H_i/N) $. Since $ N_{G/N}(H_i/N)=\dfrac{N_G(H_i)}{N} $, we have $ G=\bigcup_{i=1}^m N_G(H_i) $. By the definition of {\it normalizer covering} number of $G$, we have $ \sigma_{ \mathfrak{n}}(G)\leq m $.
\end{proof}


If $H$ and $K$ are groups and $H$ has a normalizer covering, then $H\times K$ has a {\it normalizer covering} since $N_{H\times K}(H_i\times K_i)=N_H(H_i)\times N_K(K_i)$ for every $H_i\leq H$ and $K_i\leq K$. We notice that the converse is not true, for example consider $D_8\times D_8$. But we have the following result.

\begin{prop}\label{direct}
Let $ G=H\times K $. If $\sigma_n(G)<\infty$, then  $\sigma_n(G)\leq min\{\sigma_n(H), \sigma_n(K)\}$.
In addition, if $gcd(|H|,|K|)=1$, then $\sigma_n(G)= min\{\sigma_n(H), \sigma_n(K)\}$.
\end{prop}

\begin{proof}
The proof is similar to the proof of Lemma 4 of \cite{cohn}.


\end{proof}
In the following, we give a necessary and sufficient condition for a group to have a normalizer covering.
\begin{prop}\label{fit}
Let $G$ be a  finite  group. Then $G$ has a {\it normalizer covering} if and only if the Fitting subgroup of $G$ is a subset of the union of some normalizers of non-normal subgroups in $G$, i.e. $Fit(G)\subseteq \bigcup_{i=1}^m N_G(H_i)$ where $H_i\ntriangleleft G$ for each $i$.

\end{prop}
\begin{proof}
  {If $G$ has a normalizer covering, then clearly, $Fit(G)$ is a subset of the union of that covering. 
  	
  	Now suppose that $Fit(G)$ is a subset of a union of proper normalizers in $G$.  If $b\in G\setminus Fit(G)$, then $N_G(\langle b\rangle)\neq G$, and so $G=Fit(G)\cup \bigcup_{b\in G\setminus Fit(G)} N_G(\langle b\rangle).$ This completes the proof.
  }
\end{proof}

\begin{cor}
Let $G$ be a  finite  solvable group, and let $F=Fit(G)$ be a cyclic Hall  subgroup of $G$. Then $G$ does not have any normalizer covering.
\end{cor}

\begin{proof}
Suppose, for a contradiction, that $G$ has a normalizer covering.  By Proposition \ref{fit}, there are proper normalizers $N_G(H_i)$ for $1\leq i\leq k$ such that $F\subseteq \bigcup_{i=1}^kN_G(H_i)$. By hypothesis, $F=\langle x\rangle$ for some $x\in G$. Without loss of generality, we  may assume that $x\in N_G(H_1)$.

By Normalizer-Centralizer Theorem, there exists a monomorphim $:\frac{G}{C_G(F)}\hookrightarrow Aut(F)$. Since $G$ is solvable, $C_G(F)=F$ and since $F$ is cyclic, we have $\frac{G}{F}$ is abelian. By hypothesis, there exists a subgroup $K$ of $G$ such that $G=\langle x\rangle \rtimes K$ and $gcd(o(x),|K|)=1$. It yields that $H_1=\langle x^m\rangle\rtimes V$ where $V\leq K^g$ for some $g\in G$ and some integer $m$. Clealy, $\langle x^m\rangle\lhd G$ which implies that $K^g\leq  N_G(\langle x^m\rangle)$. Since $K^g$ is an abelian group,   $K^g\leq N_G(V)$. Hence
$K^g\leq N_G(H_1)$. It follows that $H_1\lhd G$, which  is a contradiction.

\end{proof}
In the following, we obtain some information on non-nilpotent groups having no normalizer covering. In the next section we study the structure of a nilpotent group $G$ without any normalizer covering.

\begin{cor}\label{ded}
Let $G$ be a  finite  group. If $G$ does not have a {\it normalizer covering}, then $\frac{G}{Fit(G)}$ is a Dedekind group.
\end{cor}
\begin{proof} Let $F:=Fit(G)$ and $x\in G\setminus F$. If $N_{\frac{G}{F}}(\langle xF\rangle )\neq \frac{G}{F}$, then $F\leq N_G(\langle x\rangle F)$, so from Proposition \ref{fit}, $G$ has normalizer covering, which is a contradiction. Therefore all cyclic subgroups of $\frac{G}{F}$ are normal, as wanted.

\end{proof}
We end this section with the following result.

\begin{thm}
For every prime $ p$ and positive integer $k$, there exists a group $ G $ with $ \sigma_{ \mathfrak{n}}(G)=p^k+1$.
\end{thm}
\begin{proof}
First, note that if $p>2$ and $G$ is a group of order $p^3$ with exponent $p$, then $\sigma_{ \mathfrak{n}}(G)=p+1$. Also, there is a group $G$ of order 16 such that $\sigma_{ \mathfrak{n}}(G)=3$. So we may assume that $k>1$.

We construct a group $G$ of order $p^k(p^k-1)$ such that $G$ has a minimal normal subgroup of order $p^k$ and a cyclic (maximal) subgroup of order $p^k-1$. 

Consider $G=H\rtimes_{\varphi} K$ where $H$ is the additive group of the field $F$ of order $p^k$ and $K=\langle \alpha \rangle$ is the  multiplicative group of $F$ via the homomorphism $\varphi: K\rightarrow Aut(H)$ by $\varphi(\alpha)(h)=\alpha h$ for every $h\in H$. It is easy to see that $H$ is a minimal normal group of $G$ since $H$ is the only nonzero $\varphi(\alpha)$-invariant subgroup of $H$. It yields that $K$ is a maximal subgroup of $G$. Since $K$ is not normal in $G$, we have $N_G(K)=K$ and so the number of conjugates of $K$ in $G$ is $p^k$.
Also $K$ and its conjugates must be in every covering of $G$. Since $H$ is a minimal normal (abelian) subgroup of $G$ of order greater than $p$, we see that $H\subseteq N_G(T)\subset G$ for every subgroup $T$ of $H$ of order $p$. By Lemma 17 of \cite{cohn}, $G$ is covered by $p^k$ conjugates of $K$ and $N_G(T)$ which implies that $\sigma_{ \mathfrak{n}}(G)=p^k+1$. This completes the proof.

\end{proof}

\section{ normalizer coverings of nilpotent groups}
We know that $Fit(G)=G$ for every finite nilpotent group $G$. So Proposition \ref{fit} does not give any structural information about nilpotent groups having a normalizer covering. In this section, we give a sufficient condition for a nilpotent group to have a normalizer covering. At first, we investigate the normalizer coverings in $p$-groups which are the special case of the nilpotent groups.

It is well known that if $G$ is a $p$-group, then $\varPhi(G)=G^pG'$ where $G^p=\langle g^p:~~g\in G\rangle$.
Let $p>2$. A $p$-group $G$ is said to be powerful if $G^p=\varPhi(G)$, or equivalently $G'\leq G^p$. Recall that a Dedekind group of odd order is abelian by 5.3.7 of \cite{Rob}. We use of the above argument in the proof of the following lemma.
\begin{lemma}\label{gen}
Let $G$ be a   finite $p-$group with $p>2$ such that $Z(G)\leq \Phi(G)$. If $\sigma_{ \mathfrak{n}}(G)=\infty$, then   $|G:\Phi(G)|=p^2$ or equivalently $G$ is a 2-generator group. In addition, there exist $x,y\in G$ such that $G=\langle x\rangle \langle y\rangle$ where $\langle x\rangle \lhd G$ and  $\langle h\rangle\lhd G$ for any $h\in C_G(x)$.

\end{lemma}
\begin{proof}
  {   We proceed  by induction on $|G|$. The base step is clear whenever $|G|=p^3$. So suppose that $|G|>p^3$. Since $\sigma_{ \mathfrak{n}}(G)=\infty$, by Lemma \ref{nocover} there exists $x\in G$ with the following property:

  \begin{center}
  if $x\in N_G(H)$ for some subgroup $H$ of $G$, then $H\unlhd G$.~~~(*)
  \end{center}

The proof is done in four steps:

{\bf Step 1}: We claim that $\langle x\rangle\lhd G$  and $x\notin Z(G)$.

Since $x\in N_G(\langle x \rangle)$, by (*) we have $\langle x \rangle\lhd G$ . By hypothesis, $G$ is non-abelian and since $|G|$ is odd, $G$ is not Dedekind. So $G$ has a non-normal subgroup $T$. By (*), we have $x\notin N_G(T)$ and so $x\notin Z(G)$.

{\bf Step 2}: Now, we claim that $G$ is a powerful group.

  Let $\Delta$ be  a minimal set of generators of $G$. By Step 1, there exits $u\in \Delta$ such that $[u,x]\neq 1$. Set $T_u=\langle x\rangle\langle u\rangle$. Since $\langle x\rangle\lhd T_u$ and $\frac{T_u}{\langle x\rangle}$ is cyclic, we have $[x,u]\in T_u'\leq \langle x\rangle$.  If $\langle x\rangle = T_u'$, then $\langle x\rangle\leq \varPhi(T_u)$  and hence $T_u$ is cyclic, which is a contradiction. It follows that $T_u'$ is a proper subgroup of $\langle x\rangle$,  and so   $[x,u]\in \langle x^p\rangle$ (note that $\langle x^p\rangle$ is the unique maximal subgroup of $\langle x\rangle$). Therefore  $x\in N_G(\langle x^p\rangle\langle u\rangle)$ which implies that   $\langle x^p\rangle\langle u\rangle\lhd G$ by (*).
Hence every cyclic subgroup of $\frac{G}{\langle x^p\rangle}$ is normal and so $G'\leq \langle x^p\rangle \leq G^p$. Thus $G$ is a powerful group.

{\bf Step 3}: If $x\notin  \varPhi(G)$, then we may assume that $x\in \Delta$. But if $x\in \varPhi(G)$, then we find $g\in G\setminus \varPhi(G)$ such that $x=g^{p^m}$ for some positive integer $m$  and so $g$ does not belong to any proper normalizer of some subgroup of $G$ by $(*)$. Therefore we can replace $x$ with $g$ in $(*)$ and without loss of generality, we assume that $x\in \Delta$ in the sequel. Now we want to find such $g\in G$. 

By the last line of Step 2, we know that $\frac{G}{\langle x^p\rangle}$ is abelian. It follows that $x\langle x^p\rangle\in \frac{G^{p^t}\langle x^p\rangle}{\langle x^p\rangle}\setminus \frac{G^{p^{t+1}}\langle x^p\rangle}{\langle x^p\rangle}$ for some integer $t> 1$. Since 
$$\frac{G^{p^t}\langle x^p\rangle}{\langle x^p\rangle}=\{(x_1x_2...x_e)^{p^t}\langle x^p\rangle:x_1,...,x_k\in \Delta\},$$ there exist $x_1,...,x_c\in \Delta$ such that
 $x\langle x^p\rangle=(x_1x_2...x_c)^{p^t}\langle x^p\rangle$.
 Then $x=(x_1x_2...x_c)^{p^t}x^{sp}$ for some integer $s$. Hence $x^{1-sp}= (x_1x_2...x_c)^{p^t}$.  Since $gcd(p,1-sp)=1$, we have $gcd(o(x),1-sp)=1$ which follows that $\langle x^{1-sp}\rangle=\langle x\rangle$. Thus there exists an integer $r$ such that $p\nmid r$ and 
  $x=(x^{1-sp})^r=(x_1x_2...x_c)^{rp^t}$. If $x_1x_2...x_c=h\in \varPhi(G)$, then $x_1=(x_2...x_c)^{-1}h$, and so $\Delta$ is not a minimal set of generators, which is a contradiction. Hence $g:=(x_1x_2...x_c)^r\notin \varPhi(G)$ and so $x=g^{p^t}$, as wanted. 

{\bf Step 4}: Finally we show that $|\Delta|=2$ and this completes the proof.
   By Step 1, there exists $y\in \Delta$ such that $[x,y]\neq 1$. By $(*)$, we have $H\lhd G$ for every subgroup $H$ of $C_G(x)$. Therefore  $C_G(x)$ is a Dedekind group, and then it is aeblian.   
  Since $x$ is in a generator set of $C_G(x)$, there exists a subgroup $K$ of $C_G(x)$ such that $C_G(x)=K\times \langle x\rangle$. Since $p>2$,   $Aut(\langle x\rangle)$ is a cyclic group, and so
  by $N/C$ Theorem there exists $z\in G\setminus C_G(x)$ such that $G=C_G(x)\langle z\rangle=K\langle x\rangle \langle z\rangle$.
  Let $h\in K$. Since $\langle h\rangle\lhd G$ and $\frac{G}{\langle x\rangle}$ is an abelian group,  $[h,z]\leq \langle h\rangle$, and $[h,z]\in \langle x\rangle$.
  Consequently, $[h,z]\in \langle h\rangle\cap \langle x\rangle\leq K\cap \langle x\rangle=1.$
  Hence $K\leq C_G(z)$, and so 
  $K\leq Z(G)$. Therefore $G=K\langle x\rangle \langle z\rangle=\langle x\rangle \langle z\rangle$, because  $K\leq Z(G)\leq\Phi(G)$.
}
  \end{proof}

\begin{cor} \label{suf}
	Let $G$ be a   finite $p-$group with $p>2$   such that $Z(G)\leq \Phi(G)$. Then  $\sigma_{ \mathfrak{n}}(G)=\infty$ if and only if there exist $x,y\in G$ such that $G=\langle x\rangle \langle y\rangle$ where $\langle h\rangle\lhd G$ for any $h\in C_G(x)$.

\end{cor}
\begin{proof}
{First suppose that  there exist $x,y\in G$ such that $G=\langle x\rangle \langle y\rangle$ where $\langle h\rangle\lhd G$ for any $h\in C_G(x)$.  We claim that for any subgroup $H$ of $G$ if $x\in N_G(H)$, then $H\lhd G$. We proceed by induction on $|G|$. The base of induction is trivial. Let $H\leq G$ such that $x\in N_G(H)$. We may assume that $H\nleq \langle x\rangle$. Let $N:=Z(G)\cap H$. 
Suppose that there exists $gN\in Z(\frac{G}{N})\setminus \Phi(\frac{G}{N})$. Then  $\frac{G/N}{Z(G/N)}$ is a cyclic group and so
$Z(\frac{G}{N})=\frac{G}{N}$. Hence $H\lhd G$. So suppose that $Z(\frac{G}{N})\leq \Phi(\frac{G}{N})$.
If $N\neq 1$, then by induction hypothesis, $\frac{H}{N}\lhd \frac{G}{N}$, and so $H\lhd G$.
  So suppose that $N=1$. It follows that   $H\cap \langle x\rangle=1$, and then  $$H\cong \frac{H}{H\cap \langle x\rangle}\cong \frac{H\langle x\rangle}{\langle x\rangle}\leq \frac{G}{\langle x\rangle}$$
is a cyclic group.
Therefore $H=\langle x^iy^j\rangle$ for some integers $i$ and $j$. We have $[x^iy^j,x]=[y^j,x]\in H\cap \langle x\rangle=1$.
Therefore $y^j\in C_G(x)$, and so $H\leq C_G(x)$. By our assumption $H\lhd G$, as claimed.  From Lemma \ref{nocover}, $\sigma_n(G)=\infty$. 
 The proof of the other side  is clear by Lemma \ref{gen}.
}

\end{proof}
Note that the assumption $\langle h\rangle\lhd G$ for any $h\in C_G(x)$ in the previous corollary is necessary, because 
there exists a 3-group \break $G=AllSmallGroups(3^5,IsAbelian,false)[19]$  in the software $GAP$, such that $\sigma_{ \mathfrak{n}}(G)=4$.
The following result is an immediate consequence of Lemma \ref{gen}. 

\begin{cor}\label{sufficient}
	Let $G$ be a   finite $p-$group with $p>2$   such that $Z(G)\leq \Phi(G)$. If $|G:\Phi(G)|=p^{\alpha}$  and $\alpha>2$, then  $\sigma_{ \mathfrak{n}}(G)<\infty$.

\end{cor}

	


In the following lemma, we investigate 2-groups $G$  such that $Z(G)\leq \Phi(G)$ and $\sigma_n(G)=\infty$. More precisely we show that $\sigma_n(G)=\infty$ if and only if there is $x\in G$ such that $x$ satisfies in Lemma \ref{nocover}(ii), $x\notin \varPhi(G)$ and $|\frac{G}{\varPhi(G)}|\leq 16$.

\begin{lemma}\label{gen2}
	Let $G$ be a   finite $2-$group such that $Z(G)\leq \Phi(G)$. Then  $\sigma_n(G)=\infty$ if and only if  there exists $x\in G\setminus \varPhi(G)$ such that $x$ does not belong to any proper normalizer in $G$ and $G=\langle x\rangle  \langle h\rangle\langle y\rangle \langle z\rangle$ for some $h, y, z\in G$. In this case, we have $|\frac{G}{\varPhi(G)}|\leq 16$.
\end{lemma}
\begin{proof}
{
From Lemma \ref{nocover},   $(\Leftarrow)$ is clear. Suppose that $\sigma_n(G)=\infty$. Then there is $x\in G$ with the following property:

\begin{center}
	if $x\in N_G(H)$ for some $H\leq G$, then $H\unlhd G$~~~~~~~~~~~(*).
\end{center}
Without loss of generality, we may assume that $\langle x\rangle$ is a maximal cyclic subgroup of $G$. If $x\in Z(G)$, then $G$ is Dedekind by (*) and since $Z(G)\leq \varPhi(G)$, we have $G\cong Q_8$ which is a contradiction by maximality of $\langle x\rangle$. Therefore $x\notin Z(G)$ and so $o(x)\geq 4$. Now we show that $\frac{G}{\langle x^2\rangle}$ is Dedekind or equivalently $\langle g\rangle\langle x^2\rangle$ is normal in $G$ for every $g\in G$.

Assume that $g\in G$. If $[g, x]=1$, then $g\in C_G(x)$ and so $\langle g\rangle \lhd G$ by (*). It follows that $\langle g\rangle\langle x^2\rangle \lhd G$, as wanted. Now suppose that $[g, x]\neq 1$. Then $\langle g\rangle \langle x\rangle$ is non-abelian. Since $G$ is a 2-group and $\langle x\rangle \lhd G$, we have $[g, x]\in \langle x\rangle \cap \varPhi(\langle g, x \rangle)$. If $\langle [g, x]\rangle=\langle x\rangle$, then $\langle g, x\rangle=\langle g\rangle$, a contradiction. Hence $\langle [g, x]\rangle$ is proper in $\langle x\rangle$ and so $[g, x]\in \langle x^2\rangle$. It follows that $x\in N_G(\langle g\rangle\langle x^2\rangle)$ and so  $\langle g, x^2\rangle\lhd G$ by (*), as wanted.

If $\frac{G}{\langle x^2\rangle}$ is abelian, then 
$x\langle x^2\rangle\not\in \varPhi(\frac{G}{\langle x^2\rangle})=\frac{\varPhi(G)}{\langle x^2\rangle}$ by maximality of $\langle x\rangle$ and so $x\notin \varPhi(G)$.
 
If  $\frac{G}{\langle x^2\rangle}$ is non-abelian, then there are normal subgroups $A$ and $Q$ of $G$ such that $\frac{G}{\langle x^2\rangle}=\frac{A}{\langle x^2\rangle}\frac{Q}{\langle x^2\rangle}$ where $\frac{A}{\langle x^2\rangle}$ is elementary abelian, $\frac{Q}{\langle x^2\rangle}\cong Q_8$ and $A\cap Q\leq \langle x^2\rangle$. It follows that $\varPhi(\frac{G}{\langle x^2\rangle})=\varPhi(\frac{Q}{\langle x^2\rangle})=Z(\frac{Q}{\langle x^2\rangle})$. Since $\frac{\langle x\rangle}{\langle x^2\rangle}$ is a maximal cyclic subgroup of $\frac{G}{\langle x^2\rangle}$ and $\varPhi(\frac{G}{\langle x^2\rangle})=\frac{\varPhi(G)}{\langle x^2\rangle}$, we have $x\notin \varPhi(G)$.

Note that all subgroups of $C_G(x)$ are normal in $G$ by (*). If $C_G(x)$ is not an abelian group, then  $C_G(x)=D\times Q_8$ where $D$ is an elementary abelian $2$-group. Since $o(x)\geq 4$, we have  $x\not\in Z(C_G(x))$, which is a contradiction.
Therefore $C_G(x)$ is an abelian group.
There exists a subgroup $K$ such that $C_G(x)=K\times \langle x\rangle$. Since $\frac{G}{C_G(x)}$ is isomorphic to a subgroup of $Aut(\langle x\rangle)$, there exists $y,z\in G$ such that 
$G=C_G(x)\langle y\rangle \langle z\rangle$.
If $K\leq \Phi(G)$, then $G=\langle x\rangle \langle y\rangle \langle z\rangle$ and so $|\frac{G}{\varPhi(G)}|\leq 8$, as wanted.

Now suppose that $K$ is not a subgroup of $\Phi(G)$. If $ \frac{G}{\langle x^2\rangle}$ is abelian, then  $[g,K]\leq \langle x\rangle \cap K=1$ for any $g\in G$, and so $K\leq Z(G)\leq \Phi(G)$, which is a contradiction. Therefore $\frac{G}{\langle x^2\rangle}$ is non-abelian. We have been    $\frac{G}{\langle x^2\rangle}=\frac{A}{\langle x^2\rangle} \frac{Q}{\langle x^2\rangle}$ where  $\frac{A}{\langle x^2\rangle}$ is elementary abelian and $\frac{Q}{\langle x^2\rangle}\cong Q_8$.

If there exists $  a\in A\cap K$ such that $a\langle x^2\rangle\neq \langle x^2\rangle$, then $a\langle x^2\rangle\in Z(\frac{G}{\langle x^2\rangle})$, so   $[a,g]\in \langle a\rangle \cap \langle x^2\rangle=1$ for all $g\in G$.  It follows that $a\in Z(G)$. Since any non-trivial element of $\frac{A}{\langle x^2\rangle}$ is a generator of $G$, we have $a\not\in \varPhi(G)$, which is a contradiction.
Hence $K\cap A=1$. 
We have $|\frac{K\langle x^2\rangle}{\langle x^2\rangle}|=|K|$, because $K\cap \langle x\rangle =1$. If $|K|>8$, then 
$$|\frac{AK}{\langle x^2\rangle}|=|\frac{A}{\langle x^2\rangle}||\frac{K\langle x^2\rangle}{\langle x^2\rangle}|>|\frac{A}{\langle x^2\rangle}|8=|\frac{G}{\langle x^2\rangle}|,$$
which is a contradiction.
So   $|K|\leq 8$. Let $h\in K\setminus\varPhi(G)$. If $o(h)=2$, then $h\in Z(G)$, which is a contradiction. So $o(h)\geq 4$.
Since $exp(\frac{G}{\langle x^2\rangle})=4$ and $K\cap \langle x^2\rangle=1$,  either $K\cong C_4\times C_2$ or 
$K\cong C_4$.
If $K\cong C_4\times C_2$, then 
there exists $u\in K\setminus\varPhi(K)$ such that $o(u)=2$. Then $u\in Z(G)$, and so $u\in \varPhi(G)$.
There exists $q$ in $G$ such that $[q,h]\neq 1$.
Clearly, $[q,h]\not\in \langle x^2\rangle$, so
$\frac{\langle q,h,x^2\rangle}{\langle x^2\rangle}\cong Q_8$. Then $\varPhi(\frac{G}{\langle x^2\rangle})=\varPhi(\frac{\langle q,h,x^2\rangle}{\langle x^2\rangle})=\langle h^2\langle x^2\rangle\rangle$.
So $u\in h^2\langle x^2\rangle$.  It follows that
 $u=h^2x^{2i}\in K$ for some integer $i$. Since $K\cap \langle x\rangle=1$, we deduce that $u=h^2\in \varPhi(K)$, which is a contradiction. 
So $K=\langle h\rangle\cong C_4$. 
We have    $C_G(x)=\langle x\rangle\times \langle h\rangle$, and so
$$G=C_G(x)\langle y\rangle \langle z\rangle=\langle x\rangle  \langle h\rangle\langle y\rangle \langle z\rangle.$$

}
\end{proof}

Note that there exists a $2$-group $$G=AllSmallGroups(128,IsAbelian,false)[1887]$$  in the software $GAP$, such that $\sigma_n(G)=\infty$ and $|\frac{G}{\varPhi(G)}|=2^4$.
The following three proposition shows that any proper normlizer in finite $p$-group is a subgroup of a proper normalizer of   index $p$.

\begin{prop}
 \label{index}
Let $G$ be a non-abelian $p$-group and $K$ be a  normal subgroup of $G$. Suppose that $L\leq T$ are subgroups of $K$ such that $L\lhd G$, $\frac{K}{L}\cong C_p\times C_p$ and $T$ is a non-normal subgroup of $G$ such that $|K:T|=p$. Then $|G:N_G(T)|=p$.
   
\end{prop}

\begin{proof}
 By hypothesis, $T$ is a maximal subgroup of $K$ and since $K\unlhd G$, every conjugate of $T$ in $G$ is contained in $K$. Since $\frac{K}{L}\cong C_p\times C_p$, the number of maximal subgroups of $\frac{K}{L}$ is $p+1$. Therefore the number of conjugates of $T$ in $G$ is at most $p+1$ and so $|G:N_G(T)|=p$, as desired. 
\end{proof}
  
\begin{lemma}
	Let $G$ be a finite $p$-group and $H$ a normal subgroup of $G$. Then for any divisor $p^k$ of $|H|$, $H$ has a   subgroup  $K$, such that  $K\lhd G$.
	\end{lemma}  
\begin{proof}
We proceed by induction on $|G|$. Let $N\leq H$, be a normal minimal subgroup of $G$.
If $k=1$, then $K=N$. So suppose that $k>1$. By induction hypothesis, $\frac{H}{N}$ has a 
subgroup $\frac{K}{N}$ such that $\frac{K}{N}\lhd \frac{G}{N}$.	
\end{proof}	  

\begin{lemma}\label{max2}
Let $G$ be a  finite non-abelian $p-$group, and let $H$ be a non-normal subgroup of $G$. Then there exists a subgroup $L$ of $G$ such that
$N_G(H)\leq N_G(L)$ and $N_G(L)$ is a maximal subgroup of $G$.

\end{lemma}
\begin{proof}
{ We proceed by induction on $|G|$. Base of induction is trivial whenever $|G|=p^3$. So suppose that $|G|>p^3$. 
Let $$\mathcal{A}=\{K\lhd N_G(H):~~K\ntriangleleft G\}$$
 Since $H\in \mathcal{A}$, there is a subgroup $T\in \mathcal{A}$ of minimum size.

 Let $N$ be a minimal normal  subgroup of $G$. First suppose  that $\frac{TN}{N}$ is not a normal subgroup of 
$\frac{G}{N}$. By induction hypothesis, there exists a subgroup $\frac{L}{N}$ of $\frac{G}{N}$ such that
$N_{\frac{G}{N}}(\frac{TN}{N})\leq N_{\frac{G}{N}}(\frac{L}{N})$ and $N_{\frac{G}{N}}(\frac{L}{N})$ is a maximal subgroup of $\frac{G}{N}$.
 Then $N_G(L)$ is a maximal subgroup  of $G$ and $N_G(H)\leq N_G(T)\leq N_G(TN)\leq N_G(L)$. So suppose that $TN\lhd G$. Since $T\lhd N_G(T)$, $T$ has a maximal subgroup 
 $S$ which is normal in $N_G(T)$.  
 Since $|S|<|T|$, we have $S\lhd G$. It follows from Lemma \ref{index} that
 $|G:N_G(T)|=p$.  
 
}
\end{proof}
\begin{lemma}\label{cen}
Let $G$ be a finite $p$-group, and let $b\in G\setminus Z(G)$ such that 
$[b,G]\subseteq N$ where $N$ is a normal minimal subgroup of $G$.
Then $C_G(b)$ is a maximal subgroup of $G$.

\end{lemma}
\begin{proof}
 Consider the homomorphism $\sigma: G\longrightarrow Aut(\langle b\rangle N)$ by $\sigma(w)(r)=r^w$ for each $w\in G$ and $r\in \langle b\rangle N$. We  have $r^w=rc_r$ for some $c_r\in N$ and so $|Im \sigma|=p$, because $N=p$. Therefore $|\frac{G}{C_G(b)}|=p$ which implies that $C_G(b)$ is a maximal subgroup of $G$.
\end{proof}
In the following Theorem we show that a metacyclic group with derived subgroup of prime order does not have any normalizer covering.
\begin{thm}\label{meta}
    Let $K$ be a finite metacylic $p$-group such that $|K'|=p$.
    Then $\sigma_n(K)=\infty$.
\end{thm}
\begin{proof}
    Let $h$ be a generator of $K$ such that   $\langle h\rangle\lhd K$.
    Since $|K'|=p$, we have $Z(K)=\varPhi(K)$.
    If all generators of $K$ are normalal, then $K$ is abelian or quaternion, so $\sigma_n(K)=\infty$. 
So suppose that there exists    $b\in K\setminus \varPhi(K)$ such that
 $\langle b\rangle$ is not normal in $K$. Then $K=\langle h\rangle\langle b\rangle$. 
 		Since  $\langle h\rangle \lhd K$, then 
 		$K'\leq \langle h\rangle$.  Therefore $\langle b\rangle\cap \langle h\rangle=1$.   If $o(h)=p$, then $\langle h\rangle=K'\leq Z(K)$, which is a contradiction. So $o(h)>p$, and so $K'\leq \mho_1(K)$, and hence $\mho_1(K)=\varPhi(K)=Z(K)$. 
 		Let $H\leq K$, such that $h\in N_K(H)$. 
   We claim that $H\lhd K$. We proceed by induction on $|K|$.
If $H\cap \langle h\rangle\neq 1$, then $K'\leq H$, so $H\lhd K$.
So suppose that $H\cap \langle h\rangle=1$. If $[H,h]\neq 1$, then $K'\leq H$, hence $H\lhd K$. It follows that  $[H,h]=1$.

		Since $$\frac{H}{H\cap \langle h\rangle}\cong \frac{H\langle h\rangle}{\langle h\rangle}\leq \frac{K}{\langle h\rangle}$$
 		and $\frac{K}{\langle h\rangle}=\langle b\langle h\rangle\rangle$ is cyclic, we have $\frac{H}{H\cap \langle h\rangle}$ is a cyclic group, so  $H=(H\cap \langle h\rangle )\langle a_H\rangle=\langle a_H\rangle$ where $a_H\in C_K(h)=\langle h\rangle \langle b^p\rangle$. If $a_H\in Z(K)$, then $H\lhd K$.
   So $a_H=h^iz$ where $p\nmid i$ and $z\in \mho_1(K)$. If $o(a_H)=p$,
   then $h^{ip}z^p=1$. So $h^{ip}=z^{-p}\in \mho_2(K)$, so $h^p\in \mho_2(K)$,  which is a contradiction. 
   If $a_H^p\neq 1$, then  by induction hypothesis, $\frac{H}{\langle a_H^p\rangle}\lhd \frac{K}{\langle a_H^p\rangle}$.
   So $H\lhd K$, as claimed. 
 		From Lemma \ref{nocover}, $\sigma_n(K)=\infty$.
   
\end{proof}
We need   the following.

The following theorem prove that there is no any $p$-group with $\sigma_{ \mathfrak{n}}(G)=4$
 \begin{thm}\label{4}
     Let $G$ be a finite $2$-group such that $\sigma_n(G)\leq 4$. Then $\sigma_n(G)=3$.
 \end{thm}
 \begin{proof}

 Suppose, for a contradiction, that $\sigma_{ \mathfrak{n}}(G)=4$. Let  $K_1,K_2,K_3,K_4$ be distinct nonnormal subgroups of  $G$  such that $G$ is a union of their normalizers.
 By Lemma \ref{max2}, we may assume that all $R_i:=N_G(K_i)$ are maximal subgroups of $G$.
 If there are $K_{i_1}, K_{i_2}, K_{i_3}\in \{K_1,K_2,K_3,K_4\}$ such that $|G:\bigcap_{i=1}^3 R_{j_i}|=4$, then
 $$|R_{i_1}\cup R_{i_2}\cup R_{i_3}|= 3\cdot 2^{n-1}-3\cdot 2^{n-2}+ 2^{n-2}=2^n,$$  
so $G=R_{i_1}\cup R_{i_2}\cup R_{i_3}$, which is a contradiction. Therefore
  $|G:\bigcap_{i=1}^3 R_{j_i}|= 8$ for all  combinations of three subgroups  from the set $\{K_1,K_2,K_3,K_4\}$.
  Since $|R_1\cup R_2\cup R_3\cup R_4|$  is equal to
  $$\sum |R_i|-\sum_{1\leq i<j\leq 4}|R_i\cap R_j|+\sum_{1\leq i< j<s\leq 4}|R_i\cap R_j\cap R_s|-|R_1\cap R_2\cap R_3\cap R_4|,$$
  we have
  \begin{eqnarray*}
  2^n=|G|&=&|R_1\cup R_2\cup R_3\cup R_4|
  \\&<&
  4\cdot 2^{n-1}-(^4_2)\cdot2^{n-2}+(^4_3)\cdot2^{n-3}\\&=&2^{n+1}-6\cdot2^{n-2}+2^{n-1}\\&=&2^{n-1}+2^{n-1}\\&=&2^n,
  \end{eqnarray*}

  which is a contradiction.
So $\sigma_n(G)=3.$
  \end{proof}



\section{ An application }
In this section we see that by imposing  more conditions on the elements  of a normalizer covering of a group $G$, we get more information about structure of $G$.
We start this section with the following definition.
\begin{defn}
Let $H$ be a non-normal nilpotent subgroup of $G$. We say that $N_G(H)$ is a maximal n-normalizer if there is no non-normal nilpotent subgroup $L$ of $G$ such that $N_G(H)<N_G(L)$. 
\end{defn}
 
\begin{lemma}\label{hal}
Let $G$ be a finite group and $H\leq G$ such that $N_G(H)$ is a nilpotent maximal n-normalizer. If $C=Core_G(N_G(H))$, then 
  $\frac{N_G(H)}{C}$ is a Hall subgroup of $\frac{G}{C}$.

\end{lemma}
\begin{proof}
Without loss of generality, assume that $C=1$. Then for every prime divisor $p$ of $|N_G(H)|$,  the Sylow $p$-subgroup $P$ of $N_G(H)$ is not normal in $G$. We claim that  $P$ is a Sylow $p$-subgroup of $G$.

First, note that there  is a Sylow $p$-subgroup $Q$ of $G$ such that $P\subseteq Q$. Since $N_G(H)$ is nilpotent, we have $N_G(H)\leq N_G(P)$ and so by maximality,  $N_G(H)=N_G(P)$.
 Now, if $P\neq Q$, then there exists $x\in N_Q(P)\setminus P$ and so $x\in N_G(H)\setminus P$. On the other hand $x$ is a $p$-element and $N_G(H)$ is nilpotent, we have $x\in P$ which is a contradiction. Hence $P=Q$, as claimed. Hence we show that every Sylow subgroup of $N_G(H)$ is a Sylow subgroup of $G$ and this completes the proof.
\end{proof}

\begin{cor}\label{ee}
Let $G$ be a non-nilpotent group with a covering $\Gamma$ whose members are maximal n-normalizers. Also, suppose that every member of $\Gamma$ is either in $Fit(G)$ or a $q$-subgroup for a fixed prime number $q$. Then $\frac{G}{Core_G(S)}=\frac{Fit(G)}{Core_G(S)}\rtimes \frac{S}{Core_G(S)}$ is Frobenius where $S\in Syl_q(G)\cap \Gamma$. Moreover $|\Gamma|\geq1+|Syl_q(G)|$ or $\sigma_{ \mathfrak{n}}(Fit(G))+|Syl_q(G)|\}$  such that $\sigma_{ \mathfrak{n}}(Fit(G))\in\{p+1, 5\}$ where $p$ is a prime divisor of $|Fit(G)|$.

\end{cor}

\begin{proof} 
Assume that   $\Gamma=\{N_G(H_1), N_G(H_2),\cdots, N_G(H_m)\}$ for some subgroups $H_i$ of $G$. Since $G$ is not a nilpotent group and $G=\bigcup_{i=1}^mN_G(H_i)$, there exists $j$ such that  $N_G(H_j)\nless Fit(G)$. By hypothesis, $N_G(H_j)$ is a $q$-subgroup and by Lemma \ref{hal}, $\frac{N_G(H_j)}{C}$ is a Sylow $q$-subgroup of  $\frac{G}{C}$ where  $C=Core_G(N_G(H_j))$. Therefore  $N_G(H_j)\in Syl_q(G)$. Set $S:=N_G(H_j)$. Hence  $S$ is not normal in $G$ and by maximality of  $N_G(H_j)$, we have  $N_G(S)=S$.


By hypothesis, we have  $\frac{G}{C}=(\bigcup_{g\in G}\frac{S^g}{C})\cup \frac{Fit(G)}{C}$ . Since  $gcd(|\frac{S}{C}|, |\frac{Fit(G)}{C}|)=1$, we see that  $\frac{G}{C}$ has a Hall covering (for the definition of Hall covering, see \cite{hal22}). Also  since $\frac{Fit(G)}{C}$ is a normal Hall subgroup of $\frac{G}{C}$, By Zassenhaus's Theorem   $\frac{G}{C}=\frac{Fit(G)}{C}\rtimes \frac{S}{C}$. By hypothesis, $\frac{G}{C}$ does not have an element of order $pq$ for some rime divisor $p$ of $|Fit(G)|$. Therefore
the prime graph of $\frac{G}{C}$ is not connected. It follows  from Theorem B of \cite{hal22} that $\frac{G}{C}$  is a Frobenius group or $2$-Frobenius group. Now, we claim that  $\frac{G}{C}$ is Frobenius.

Suppose, for a contradiction, that  $\frac{G}{C}$ is a 2-Frobenius group.  Then $G$ has a normal series  $C<N<K<G$ such that  $\frac{G}{N}$ and  $\frac{K}{C}$ are Frobenius. Therefore  $\frac{G}{N}=Fit(\frac{G}{N})\rtimes \frac{SN}{N}$ and $\frac{K}{C}=\frac{Fit(G)}{C}\rtimes \frac{S\cap K}{C}$. Now if  $S\cap S^g\neq C$ for some $g\in G\setminus S$, then there is  $xC\in\frac{S\cap S^g}{C}$ of order $q$.  On the other hand  we know  $SN\cap S^gN=N$ and so  $S\cap S^g<K$. Since $\frac{S\cap K}{C}$ is cyclic or generalized Quaternion,   $\langle xC\rangle$ is characteristic subgroup of $\frac{S\cap K}{C}$ and $\frac{S^g\cap K}{C}$. Therefor  $\langle xC\rangle$ is a normal subgroup of $\frac{S}{C}$ and  $\frac{S^g}{C}$ which implies that $S, S^g\leq N_G(\langle x\rangle C)$. Since $S$ is a maximal n-normalizer in $G$ and  $\langle x\rangle C$ is not normal in $G$, we have a contradiction. It follows that  $S\cap S^g=C$ for every $g\in G$. Now since  $\frac{G}{C}=\frac{Fit(G)}{C}\rtimes \frac{S}{C}$, we see that  $\frac{G}{C}$ is Frobenius which is impossible. Hence  $\frac{G}{C}$ is not 2-Frobenius and so $\frac{G}{C}=\frac{Fit(G)}{C}\rtimes \frac{S}{C}$ is Frobenius. It follows that all conjugates of $S$ in $G$ are in $\Gamma$.

Now, suppose that $N_G(L_1), N_G(L_2), \cdots, N_G(L_r)$ are all members of $\Gamma$ which are contained in $Fit(G)$. Then  we show that  $\bigcup_{i=1}^rN_G(L_i)= Fit(G)$.

Suppose, for a contradiction,  that $\bigcup_{i=1}^rN_G(L_i)\neq Fit(G)$. Then $C$ is not  a subset of $\bigcup_{i=1}^rN_G(L_i)$. Since $L_i$ is a non-normal nilpotent subgroup of $G$ for each $i$,  there exists $R_i\in Syl_{r_i}(L_i)$  such that $R_i$ is not a normal subgroup of $G$ for some prime divisor $r_i$ of $|L_i|$. Since $R_i$ is a characteristic subgroup of $L_i$ and $L_i\lhd N_G(L_i)$, we have $R_i\lhd N_G(L_i)$ for each $i$. Hence $N_G(L_i)\leq N_G(R_i)$. Since $N_G(L_i)$ is maximal n-normalizer and $N_G(R_i)\neq G$, we have $N_G(R_i)=N_G(L_i)$. If $R_i$ is not a Sylow $q$-subgroup of $L_i$, then $C\leq C_G(R_i)$ since  $C$ and $R_i\leq Fit(G)$. Therefore $C\leq N_G(R_i)=N_G(L_i)$, which is a contradiction. Hence   $R_i\in Syl_q(L_i)$ for all $i=1,2,...,r$. Let $Y$ be a Hall complement for $C$ in $Fit(G)$. Since $Y\leq N_G(R_i)$, we have $Y\leq N_G(L_i)$  for all $i=1,2,...,r$.  
Let $x\in C\setminus \bigcup_{i=1}^rN_G(L_i)$ and  $y\in Y\setminus \{1\}$.
Since  $xy\not\in C$,  there exists $L_i$ such that $xy\in N_G(L_i)$ where $1\leq i\leq r$.
Since $y\in Y\leq  N_G(L_i)$, we have $x\in N_G(L_i)$, which is a contradiction. Hence $\bigcup_{i=1}^rN_G(L_i)= Fit(G)$. If $r=1$, then $|\Gamma|=1+ |Syl_q(G)|$. Otherewise $\{N_{Fit(G)}(L_1),...,N_{Fit(G)}(L_r)\}$ is a normalizer covering for $Fit(G)$ and so $|\Gamma|= r+ |Syl_q(G)|\geq \sigma_{ \mathfrak{n}}(Fit(G))+|Syl_q(G)|$. The proof is complete by Corollary \ref{nilll}.


\end{proof}

We conclude the paper with the following questions which can be interesting:
 
 \begin{que}\label{ques}
 
 (1)-For every positive integer $m>2$, is there a group $G$ with $\sigma_{ \mathfrak{n}}(G)=m$?
 
 (2)-Determine the groups $G$ with $\sigma_{ \mathfrak{n}}(G)=\sigma(G)$?
 
 (3) Let $\Gamma=\{N_G(H_i):~1\leq i\leq m\}$ be a  normalizer covering of $G$ such that every $N_G(H_i)$ is nilotent and there is no a nilpotent subgroup $K_i$ such that $N_G(H_i)<N_G(K_i)<G$ for each $i$, then is $G$ solvable?
 
 (4) Find $\sigma_{ \mathfrak{n}}(G)$ when $G$ is a solvable group wih a normalizer covering.

 \end{que}
 In \cite{tom}, Tomkinson proved if $G$ is a non-cyclic solvable group, then $\sigma(G)=p^k+1$ for some prime $p$ and positive integer $k$. But we have seen $\sigma_{ \mathfrak{n}}(S_4)=7$ and so the situation of $\sigma_{ \mathfrak{n}}(G)$ is different from $\sigma(G)$ on solvable groups $G$. Corollary \ref{ee} gives a partial positive answer to Question \ref{ques}(3).


 \end{document}